\documentclass[11pt]{article}

\usepackage[utf8]{inputenc}
\usepackage[T1]{fontenc}
\usepackage{mathpazo}
\usepackage{needspace}

\RequirePackage[english]{babel}

\RequirePackage[a4paper,margin=2.5cm]{geometry}
\RequirePackage{parskip}

\newcounter{case}[section]
\renewcommand{\thecase}{\arabic{case}}

\newcommand{\caseheading}{%
  \refstepcounter{case}%
  \par\medskip\noindent\textbf{Case \thecase. }\ignorespaces
}

\newcommand{\affiliation}{\footnote}
\newcommand{\comment}[1]{}
\makeatletter
\def\@fnsymbol#1{\ensuremath{\ifcase#1\or *\or \dagger\or \ddagger\or \mathsection\or \|\or **\or \dagger\dagger \or \ddagger\ddagger \else\@ctrerr\fi}}
\makeatother

\usepackage{xcolor}
\definecolor{cblue}{RGB}{0,70,140}
\definecolor{cgreen}{RGB}{100,140,0}
\definecolor{cred}{RGB}{190,10,50}

\usepackage[shortlabels]{enumitem}
\setlist{itemsep=0ex,topsep=0ex,parsep=0.4ex}

\usepackage{amsmath,amsfonts,amssymb,amsthm,mathtools,thmtools,thm-restate,bbm,centernot}

\usepackage[hyphens]{url} 
\usepackage[hidelinks,colorlinks,pagebackref]{hyperref} 
\hypersetup{citecolor=cgreen,linkcolor=cblue,urlcolor=cblue}
\usepackage[capitalise,nameinlink,noabbrev,compress]{cleveref} 

\renewcommand*{\backref}[1]{}

\renewcommand*{\backrefalt}[4]{%
    \ifcase #1\relax
        Not cited.%
    \or
        #2%
    \else
        #2%
    \fi
}

\theoremstyle{plain}
\newtheorem{theorem}{Theorem}[section]
\newtheorem{lemma}[theorem]{Lemma}

\newtheorem{proposition}[theorem]{Proposition}

\newtheorem{corollary}[theorem]{Corollary}

\theoremstyle{definition}

\newtheorem{claim}[theorem]{Claim}

\makeatletter
\renewenvironment{proof}[1][\proofname]
{\par\pushQED{\qed}
	\normalfont\topsep6\p@\@plus6\p@\relax\trivlist
	\item[\hskip\labelsep\bfseries#1\@addpunct{.}]
	\ignorespaces}
{\popQED\endtrivlist\@endpefalse}
\makeatother

\newcommand{\ex}{{\rm ex}}

\title{On the Tur\'an Density of $C_{10}$ in the Hypercube}
\author{Marko Peji\'c \affiliation{Karlsruhe Institute of Technology, Germany (\textsf{\href{mailto:marko.pejic@student.kit.edu}{marko.pejic@student.kit.edu}})}}

\begin{document}
\maketitle
\hypersetup{linkcolor=cred}

\newcommand{\etalchar}[1]{$^{#1}$}

\begin{abstract}

The $n$-dimensional hypercube $Q_n$ is the graph with vertex set $\{0,1\}^n$ in which two vertices are adjacent if they differ in exactly one coordinate. 
For a graph $H$, let $\ex(Q_n,H)$ be the maximum number of edges in an $H$-free subgraph of $Q_n$. 
The \emph{hypercube Tur\'an density} of $H$ is defined by $\pi_{\square}(H)=\lim_{n\rightarrow\infty}\ex(Q_n,H)/|E(Q_n)|$. 
In this note, we prove
\[
    \frac{1}{8} \leq \pi_{\square}(C_{10}) \leq 0.36577.
\]
For the upper bound, we prove $\pi_{\square}(C_{10}) \leq \pi_{\square}(C_6)$, which, together with a result of Baber, gives the stated upper bound. 
For the lower bound, we prove that $\ex(Q_n,C_{10}) > |E(Q_n)|/8$ for every $n \geq 2$. 

\end{abstract}

\section{Introduction}

The $n$-dimensional hypercube $Q_n$ is the graph with vertex set $\{0,1\}^n$ in which two vertices are adjacent if they differ in exactly one coordinate. 
A copy of a graph $H$ in a graph $G$ is a subgraph of $G$ isomorphic to $H$. 
We say that $G$ is $H$-free if it contains no copy of $H$. 
For a graph $H$, let $\ex(Q_n, H)$ denote the maximum number of edges in a subgraph of $Q_n$ containing no copy of $H$. 
Let $\pi_\square(H) = \lim_{n \rightarrow \infty} \ex(Q_n, H)/|E(Q_n)|$ denote the hypercube Tur\'an density of $H$. 
We say that a graph $H$ has {\it zero Tur\'an density in the hypercube} if $\pi_\square(H)=0$, and {\it positive Tur\'an density in the hypercube} otherwise. 
In the following, we simply say Tur\'an density instead of Tur\'an density in the hypercube. 
A standard double counting argument shows that the sequence $\ex(Q_n, H)/|E(Q_n)|$ is non-increasing, so the limit exists and the given density notion is well-defined. 
Throughout the remainder of this note, we write $[n] = \{1, \ldots, n\}$ and $e(G) = |E(G)|$. 
For a vertex $x= (x_1, \ldots, x_n) \in \{0,1\}^n$, we call $x_i$ its
$i$-th coordinate. 
The direction of an edge $xy$ of $Q_n$ is the unique index $i\in[n]$
such that $x_i\neq y_i$, and the direction set of a subgraph of $Q_n$
is the set of directions of its edges.

A central question is to determine $\pi_\square(C_{2\ell})$, where $C_{2\ell}$ denotes the cycle of length $2\ell$.
It is well known that $\pi_\square(C_{4}) \geq 1/2$, while Baber~\cite{baber} proved the upper bound $\pi_\square(C_{4}) \leq 0.60318$. 
Further, Chung~\cite{chung} first proved $\pi_{\square}(C_6) \geq 1/4$, Conder~\cite{conder} improved this to $1/3$, and Baber~\cite{baber} proved the upper bound $\pi_{\square}(C_6) \leq 0.36577$. 
Using flag algebras, Balogh, Hu, Lidick\'y, and Liu~\cite{BHLL} obtained slightly weaker upper bounds for $C_4$ and $C_6$.
Moreover, $C_{2\ell}$ has zero Tur\'an density for every integer
$\ell \geq 2$ with $\ell \notin \{2,3,5\}$, see
Chung~\cite{chung}, F\"uredi and \"Ozkahya~\cite{FOe1,FOe2}, and
Conlon~\cite{conlon}, as well as Axenovich~\cite{axenovich} for an
overview of these results.
Hence, $C_{10}$ was the final even cycle for which it remained unknown whether $\pi_{\square}(C_{10})=0$.
Grebennikov and Marciano~\cite{GM1} settled this case by proving that
$\pi_{\square}(C_{10}) > 0.024$.\footnote{The published version gives the weaker lower bound $\pi_{\square}(C_{10}) > 0.024$, while the current arXiv version~\cite{GM2} improves this to $\pi_{\square}(C_{10}) > 0.036$.} 
Their construction is strongly inspired by a result of Ellis, Ivan, and Leader~\cite{EIL} for daisy-free hypergraphs.
Before the present note, the best upper bound on $\pi_{\square}(C_{10})$ known to us was $1/\sqrt{2} \approx 0.7071$, which follows from a result of Axenovich and Martin~\cite[Theorem~3.3]{AM}. 
In this note, we improve both the upper and lower bounds on $\pi_{\square}(C_{10})$. 
For the upper bound, we prove a relation between the Tur\'an densities of $C_6$ and $C_{10}$. 
For the lower bound, we prove a slightly stronger statement about the extremal number of $C_{10}$ in $Q_n$ and the proof is inspired by the construction of Grebennikov and Marciano~\cite{GM2}. 
That is, we shall also assign a vector $v_i$ to each $i \in [n]$ and select vertices according to basis conditions. 
However, we work over $\mathbb{F}_3$ rather than $\mathbb{F}_2$ and use a different condition for retaining edges, which makes the resulting graph $C_{10}$-free from begin with.

\begin{theorem}\label{main}
The hypercube Tur\'an densities of $C_6$ and $C_{10}$ satisfy
\[
    \pi_{\square}(C_{10}) \leq \pi_{\square}(C_6).
\]
\end{theorem}

\begin{theorem}\label{main2}
For every integer $n \geq 2$,
\[
    \ex(Q_n,C_{10}) > \frac{1}{8} e(Q_n).
\]
\end{theorem}

Combining Theorems~\ref{main} and~\ref{main2} with the result of Baber~\cite{baber} gives $1/8 \leq \pi_{\square}(C_{10}) \leq 0.36577$.

\section{Proof of Theorem~\ref{main}}

Let $H$ be a $C_{10}$-free subgraph of $Q_n$. 
We first use the $C_{10}$-free condition to bound the number of copies of $C_6$ in $H$. 
We then prove Theorem~\ref{main} by applying this bound to an averaging argument over the subcubes of $Q_n$.

Every cycle in $Q_n$ uses each direction an even number of times. Since a copy of $C_6$ cannot be contained in a subcube of dimension at most two, every copy of $C_6$ uses exactly three directions, each twice. 
Fix an edge $e \in E(H)$ with direction $i$. 
Let $\mathcal{F}_e$ be the family of pairs of distinct directions $\{j,k\} \subseteq [n] \setminus \{i\}$ such that some copy of $C_6$ in $H$ containing $e$ uses exactly the directions $i,j,k$.

\begin{claim}\label{claim: intersect}
The family $\mathcal{F}_{e}$ is intersecting. In particular, $|\mathcal{F}_{e}| \leq n-2$ for every $n \geq 5$.

Suppose otherwise, and let $\{j,k\},\{\ell,m\} \in \mathcal{F}_e$ be
disjoint. 
Write $e = \{x,y\}$. 
Then there exist copies $C$ and $C'$ of $C_6$ in $H$, both containing $e$, whose direction sets are $\{i,j,k\}$ and $\{i,\ell,m\}$, respectively. 
We claim that $V(C) \cap V(C') = \{x,y\}$. 
Otherwise, let $v \notin \{x,y\}$ be a common vertex, and let $z \in \{x,y\}$ be the unique endpoint of $e$ whose $i$-th coordinate agrees with that of $v$.
Since $C$ uses only the directions $i,j,k$, and since $v$ and $z$ have the same $i$-th coordinate, the set of indices at which they differ is a nonempty subset of $\{j,k\}$.
Since $v,z \in V(C')$ as well, the same set is also a subset of $\{\ell,m\}$. 
But this is impossible since $\{j,k\} \cap \{\ell,m\} = \varnothing$. 
Hence $C - e$ and $C' - e$ are internally vertex-disjoint paths of length
five with the same endpoints, so their union is a copy of $C_{10}$ in
$H$, a contradiction. 
Thus $\mathcal{F}_e$ is intersecting, and the Erd\H{o}s--Ko--Rado theorem~\cite{EKR} gives $|\mathcal{F}_e| \leq n-2$ for $n \geq 5$.
This proves Claim~\ref{claim: intersect}.

\end{claim}

\begin{claim}\label{claim: directions} Fix distinct directions $i,j,k \in [n]$ and an edge $e$ of direction $i$. Let $Q$ be the unique $3$-dimensional subcube containing $e$ whose direction set is $\{i,j,k\}$. Then $Q$ contains exactly eight copies of $C_6$ containing $e$. 

Let $C$ be a copy of $C_6$ in $Q$ containing $e$. 
Then $P = C - e$ is a path of length five between the endpoints of $e$. 
Since $C$ uses each of the directions $i,j,k$ twice and $e$ has direction $i$, the path $P$ uses direction $i$ once and directions $j$ and $k$ twice. 
Traverse $P$ starting from a fixed endpoint of $e$. 
Note that $P$ cannot start with direction $i$, since the unique edge of direction $i$ incident with its initial endpoint is $e$. 
Suppose first that $P$ starts with direction $j$. 
A direct check shows that the possible direction sequences along $P$ are exactly: $(j,i,k,j,k)$, $(j,k,i,j,k)$, $(j,k,i,k,j)$, and $(j,k,j,i,k)$. 
Interchanging $j$ and $k$ gives the remaining four paths. 
Adding $e$ to any of the paths gives a distinct copy of $C_6$. 
This proves Claim~\ref{claim: directions}.

\end{claim}

\begin{claim}\label{claim: N_6} For $n \geq 5$, the number $N_H$ of copies of $C_6$ in $H$ satisfies $N_H \leq \frac{4}{3}(n-2)e(H)$. In particular, $N_H \leq \frac{2}{3}n(n-2)2^n$.

Each copy of $C_6$ contributes six incidences with its edges. 
Conversely, Claims~\ref{claim: intersect} and~\ref{claim: directions} imply that every edge $e \in E(H)$ is contained in at most $8|\mathcal{F}_e| \leq 8(n-2)$ copies of $C_6$. 
Thus, by double counting, $N_H \leq \frac{4}{3}(n-2)e(H)$. 
Also, since $e(H) \leq n2^{n-1}$, it follows that $N_H \leq \frac{2}{3}n(n-2)2^n$. 
This proves Claim~\ref{claim: N_6}.

Now let $\mathcal{Q}_d(n)$ be the set of all $d$-dimensional subcubes in $Q_n$, and let $M = |\mathcal{Q}_d(n)| = \binom{n}{d}2^{n-d}$. 
A standard averaging argument, obtained by double counting pairs $(f,Q)$ with $Q \in \mathcal{Q}_d(n)$ and $f \in E(H \cap Q)$, gives 
\begin{equation}\label{equation: average}
    \frac{1}{M} \sum_{Q \in \mathcal{Q}_d(n)} \frac{e(H \cap Q)}{e(Q_d)} = \frac{e(H)}{e(Q_n)}.
\end{equation}

\end{claim}

We are now ready to complete the proof.

\phantomsection
\label{proof:main}
\begin{proof}[Proof of Theorem~\ref{main}]

Fix $d \geq 3$ and assume that $n \geq \max\{d,5\}$. 
Let $\alpha = \ex(Q_d,C_6)/e(Q_d)$.
Call a subcube $Q \in \mathcal{Q}_{d}(n)$ \emph{bad} if $H \cap Q$ contains a copy of $C_6$. 
Let $b$ be the number of bad subcubes. 
If $Q$ is bad, we use $e(H \cap Q) \leq e(Q_d)$. 
If $Q$ is not bad, we use $e(H \cap Q) \leq \alpha e(Q_d)$, since $H \cap Q$ is $C_6$-free. 
It follows that $\sum_{Q \in \mathcal{Q}_d(n)} e(H \cap Q) \leq (M-b)\alpha e(Q_d) + be(Q_d)$.
Dividing by $Me(Q_d)$ and applying ~\eqref{equation: average}, we obtain
\[
    \frac{e(H)}{e(Q_n)} \leq \alpha + (1-\alpha)\frac{b}{M} \leq \alpha + \frac{b}{M}.
\]
It remains to bound $b/M$. 
To bound $b$, we double count pairs $(C,Q)$, where $C$ is a copy of $C_6$ in $H$ and $Q \in \mathcal{Q}_d(n)$ contains $C$. 
Each bad subcube contributes at least one such pair.
Conversely, every copy of $C_6$ is contained in exactly $\binom{n-3}{d-3}$ $d$-dimensional subcubes. 
Thus, by Claim~\ref{claim: N_6}, we obtain 
\[
    \frac{b}{M} \leq N_H\frac{\binom{n-3}{d-3}}{\binom{n}{d}2^{n-d}} \leq \frac{2^{d+1}}{3}\frac{d(d-1)(d-2)}{n-1}.
\]

For fixed $d \geq 3$, the bound on $b/M$ tends to zero as $n \rightarrow \infty$. 
Since $H$ was arbitrary, taking the maximum over all $C_{10}$-free subgraphs of $Q_n$ and then letting $n \to \infty$ gives
\[
    \pi_{\square}(C_{10})
    \leq
    \frac{\ex(Q_d,C_6)}{e(Q_d)}.
\]
Recall that the sequence $\ex(Q_d,C_{6})/e(Q_d)$ is non-increasing in $d$ and thus converges to $\pi_{\square}(C_6)$ as $d \rightarrow \infty$. 
It follows that $\pi_{\square}(C_{10}) \leq \pi_{\square}(C_6)$. 
This proves Theorem~\ref{main}.

\end{proof}

\section{Proof of Theorem~\ref{main2}}

For the lower bound, it is convenient to regard each vertex of $Q_n$ as its support in $[n]$. 
For example, we regard the vertex $x = (1,0,1) \in \{0,1\}^3$ as the set $\{1,3\} \subseteq [3]$.
For every $r \in [n]$, let $L_r(n)$ be the edge layer of $Q_n$ induced by $\binom{[n]}{r-1} \cup \binom{[n]}{r}$. 
Call the vertices in $\binom{[n]}{r-1}$ the \emph{lower vertices} of $L_r(n)$ and call those in $\binom{[n]}{r}$ the \emph{upper vertices}. 
We first prove the following result for a single edge layer.

\begin{proposition}\label{proposition: layer_lower}
For all integers $n \geq 1$ and every $r \in [n]$, there is a $C_{10}$-free subgraph $H \subseteq L_r(n)$ satisfying
\[
    e(H) \geq \frac{e(L_r(n))}{4-3^{1-r}}.
\]
\end{proposition}

We shall need the following notation throughout.
Let $B=\{b_1,\ldots,b_m\}$ be a basis of an $m$-dimensional vector
space $V$ over $\mathbb{F}_3$, and let $z\in V$.
Writing $z = \lambda_1b_1 + \ldots + \lambda_mb_m$, we say that $b_j$ has \emph{coefficient one in $z$ with respect to $B$} if $\lambda_j = 1$. 
The proof of Proposition~\ref{proposition: layer_lower} relies on the following lemma, whose proof is given in the appendix.

\begin{lemma}\label{lemma: coefficient}

Let $C$ be a copy of $C_{10}$ contained in an edge layer of $Q_n$. 
Let $W \subseteq[n]$ be a maximal set such that each upper vertex of $C$
contains $W$, and write its five upper vertices in cyclic order as
$W \cup U_0, \ldots, W \cup U_4$, where $|U_i| = m$ for every $i \in \{0,\ldots, 4\}$. 
Let $V$ be any $m$-dimensional vector space over $\mathbb{F}_3$ and assume that for each $x \in U_0 \cup \ldots \cup U_4$ there is vector $v_x$ in V such that $\{v_x: x \in U_i\}$ is a basis of $V$ for every
$i \in \{0,\ldots,4\}$. 
If $p_i,q_i \in U_i$ are the directions of the two edges of $C$ incident
to $W \cup U_i$, then, for every $z \in V$, there is an $i \in \{0, \ldots, 4\}$ such that at most one of $v_{p_{i}}$ and $v_{q_{i}}$ has  coefficient one in $z$ with respect to $\{v_x : x \in U_i\}$.

\end{lemma}

\begin{proof}[Proof of Proposition~\ref{proposition: layer_lower}]

Fix a non-zero vector $\beta\in\mathbb{F}_3^r$ and choose a vector
$v_i\in\mathbb{F}_3^r$ for every $i\in[n]$.
Call a lower vertex $X \in \binom{[n]}{r-1}$ \emph{selected} if $\{\beta\} \cup \{v_i: i \in X\}$ is a basis, and call an upper vertex $Y \in \binom{[n]}r$ selected if $\{v_i: i \in Y\}$ is a basis. 
For an edge $XY \in E(L_r(n))$ with  $X$ selected and $Y = X \cup \{j\}$, write $v_j = \lambda_0 \beta + \sum_{i \in X} \lambda_i v_i$ and retain $XY$ if and only if $\lambda_0 = 1$. 
Let $F$ be the resulting graph. 
Note that every retained edge has both endpoints selected, since replacing $\beta$ by $v_j$ in the basis $\{\beta\} \cup \{v_i: i \in X\}$ gives the basis $\{v_i: i \in Y\}$.

\begin{claim}\label{claim: coefficient_free} The graph $F$ is $C_{10}$-free.

Suppose otherwise, and let $C$ be a copy of $C_{10}$ in $F$. 
Let $W \subseteq [n]$ be a maximal set such that each upper vertex of $C$ contains $W$.
Let $V = \mathbb{F}_3^r/\operatorname{span}(\{v_i: i \in W\})$, and denote the images of $\beta$ and $v_x$ in $V$ by $b$ and $w_x$, respectively. 
Consider an edge of $C$ whose upper vertex is $W \cup U$ and whose direction is $x \in U$. 
Since its lower endpoint is selected, $\{b\}\cup\{w_y: y \in U - \{x\}\}$ is a basis of $V$. 
Moreover, since the edge was retained, we have $w_x = b + \sum_{y \in U-\{x\}} \lambda_y w_y$. 
Equivalently, $b = w_x - \sum_{y \in U-\{x\}} \lambda_y w_y$. 
This implies that $\{w_y: y \in U\}$ is a basis of $V$, and $w_x$ has coefficient one in $b$ with respect to this basis. 
Applying this argument to the two edges incident with each upper vertex of $C$ contradicts Lemma~\ref{lemma: coefficient}. 
This proves Claim~\ref{claim: coefficient_free}.

\end{claim}

We now choose $v_1, \ldots, v_n$ independently and uniformly at random from $\mathbb{F}_3^r$, and let $R$ be the set of selected vertices. 
Let $c_r = \prod_{k=1}^{r-1}(1 - 3^{-k})$, where the empty product is one. 
Fix an edge $e = XY$ with $|X| = r - 1$. 
Note that $\operatorname{Pr}(X \in R) = c_r$ and $\operatorname{Pr}(Y \in R) = c_r(1 - 3^{-r})$. 
Conditional on $X \in R$, the coefficient $\lambda_0$ is uniformly distributed over $\mathbb{F}_3$, and $Y \in R$ if and only if $\lambda_0 \neq 0$, while $e \in E(F)$ if and only if $\lambda_0 = 1$. 
Hence, $\operatorname{Pr}(X \in R\text{ and }Y\in R) = 2c_r/3$ and $\operatorname{Pr}(e \in E(F))=c_r/3$. 
Applying inclusion-exclusion to the events $X \in R$ and $Y \in R$, we obtain
$\operatorname{Pr}(\{X,Y\} \cap R \neq \varnothing) = \operatorname{Pr}(X \in R) + \operatorname{Pr}(Y \in R) - \operatorname{Pr}(X \in R\text{ and }Y \in R) = c_r(4/3 - 3^{-r})$.

Now repeat the construction independently in rounds, keeping $\beta$ fixed. 
In round $t$, let $F_t$ denote the resulting graph and let $R_t$ be the set of selected vertices.
Let $H_t$ consist of the edges of $F_t$ neither of whose endpoints belong to $R_1 \cup \ldots \cup R_{t-1}$.
Hence, no vertex is incident with edges from two different graphs $H_t$. 
Now let $H = \bigcup_{t \geq 1} H_t$. 
Since each $H_t \subseteq F_t$ is $C_{10}$-free by Claim~\ref{claim: coefficient_free}, their union $H$ is also $C_{10}$-free. 
Also, for a fixed edge $e \in E(L_r(n))$, we have $\operatorname{Pr}(e\in E(H_t)) = (1 - c_r(4/3 - 3^{-r}))^{t-1} \cdot c_r/3$, since neither endpoint may be selected in the first $t-1$ rounds and $e$ is retained in round $t$. 
Summing over all rounds gives
\[
    \operatorname{Pr}(e \in E(H)) =\sum_{t=1}^{\infty}
    \left(1-c_r\left(\frac{4}{3} - 3^{-r}\right)\right)^{t-1} \cdot \frac{c_r}{3} = \frac{1}{4-3^{1-r}}.
\]
Thus $\mathbb{E}[e(H)]=e(L_r(n))/(4-3^{1-r})$, so there exists a sequence of vector assignments, one for each round, for which the resulting graph $H$ satisfies $e(H) \geq e(L_r(n))/(4-3^{1-r})$. 
This completes the proof of Proposition~\ref{proposition: layer_lower}.

\end{proof}

We now apply Proposition~\ref{proposition: layer_lower} to all odd edge layers to prove Theorem~\ref{main2}.

\begin{proof}[Proof of Theorem~\ref{main2}]

For every odd $r \in [n]$, choose a graph $H_r \subseteq L_r(n)$ given by Proposition~\ref{proposition: layer_lower}. 
Distinct odd edge layers have disjoint vertex sets, so their union is $C_{10}$-free. 
Hence, for every $n \geq 1$,
\[
    \ex(Q_n,C_{10}) \geq \sum_{\substack{1 \leq r \leq n\\r\text{ odd}}} \frac{e(L_r(n))}{4 - 3^{1-r}}.
\]
For $n \geq 2$, we have $4 - 3^{1-r} < 4$ for every $r$, and exactly half of the edges of $Q_n$ belong to the odd edge layers. 
Hence
\[
    \ex(Q_n,C_{10}) > \frac{1}{4} \sum_{\substack{1 \leq r \leq n\\r\text{ odd}}} e(L_r(n)) = \frac{1}{8}e(Q_n).
\]
This proves Theorem~\ref{main2}. 
Dividing by $e(Q_n)$ and letting $n \rightarrow \infty$ gives the claimed lower bound on $\pi_{\square}(C_{10})$.

\end{proof}

\section{Concluding Remarks}

\paragraph{A simultaneous lower bound.}

It can be shown that, for every $n \geq 2$, the graph constructed in the proof of Theorem~\ref{main2} is simultaneously $C_4$-, $C_6$-, and $C_{10}$-free.

\paragraph{Generalized Tur\'an densities in the hypercube.}

The counting argument in Claim~\ref{claim: N_6} can also be applied to generalized Tur\'an densities in the hypercube (see, e.g., Axenovich, Benz, Offner, and Tompkins~\cite{ABOT}). 
In particular, for fixed cubical graphs $T$ and $H$, let $\ex(Q_n,T,H)$ be the maximum number of copies of $T$ in an $H$-free subgraph of $Q_n$, and let $N(Q_n,T)$ be the number of copies of $T$ in $Q_n$. 
Let
\[
    \mathbf{d}(Q_n,T,H) = \frac{\ex(Q_n,T,H)}{N(Q_n,T)}.
\]
That is, $\mathbf{d}(Q_n,T,H)$ is the maximum proportion of the copies of $T$ in $Q_n$ that can be retained in an $H$-free subgraph of $Q_n$. 
Axenovich, Benz, Offner, and Tompkins~\cite[Lemma~3]{ABOT} showed $\mathbf{d}(Q_n,T,H) \leq \ex(Q_n,H)\allowbreak/e(Q_n)$. 
Thus, a result of Axenovich and Martin~\cite[Theorem~3.3]{AM} implies $\mathbf{d}(Q_n,C_{6},C_{10})\allowbreak\leq\allowbreak1/\sqrt{2} + o(1)$ as $n \rightarrow \infty$. 
We give the following improvement.

\begin{corollary}\label{generalized_corollary}
For every $n \geq 5$, 
\[
    \mathbf{d}(Q_n,C_6,C_{10}) 
    \leq \frac{2}{n-1}.
\]
In particular,
\[
    \lim_{n \rightarrow \infty} \mathbf{d}(Q_n,C_6,C_{10}) = 0.
\]
\end{corollary}

\begin{proof}[Proof of Corollary~\ref{generalized_corollary}]

Taking the maximum over all $C_{10}$-free subgraphs of $Q_n$ in Claim~\ref{claim: N_6} gives $\ex(Q_n,C_6,C_{10})\allowbreak\leq\allowbreak\frac{4}{3}(n-2)\ex(Q_n,C_{10})$.
Every $C_6$ is contained in a unique $3$-dimensional subcube of $Q_n$. 
There are $\binom{n}{3}2^{n-3}$ such subcubes, and each contains exactly $16$ copies of $C_6$. 
Thus $N(Q_n,C_6) = \frac{1}{3}n(n-1)(n-2)2^n$.
Now dividing $\ex(Q_n,C_6,C_{10})$ by $N(Q_n,C_6)$ gives
\[
    \mathbf{d}(Q_n,C_6,C_{10}) \leq \frac{2}{n-1} \frac{\ex(Q_n,C_{10})}{e(Q_n)} \leq \frac{2}{n-1}.
\]
The last inequality follows since $\ex(Q_n,C_{10}) \leq e(Q_n)$. 
This proves Corollary~\ref{generalized_corollary}. 

\end{proof}

{\bf Open Questions.}
Is $\pi_{\square}(C_{10})=1/8$? 
Is $\pi_{\square}(C_{10})=\pi_{\square}(C_6)$?

{\bf AI Disclosure.}
OpenAI's ChatGPT-5.6 Plus assisted in developing the connection between the $C_6$-counting argument and the subcube-averaging argument used in the final part of the proof of Theorem~\hyperref[proof:main]{\ref*{main}}. 
It also assisted in showing that the construction in Theorem~\ref{main2} is $C_{10}$-free. 
It was in particular used to check whether any of the five types considered in the proof of Lemma~\ref{lemma: coefficient} could occur.
Both parts were subsequently checked, shortened, and reworked by the author. 
ChatGPT-5.6 Plus was also used to check the note for misprints and typos.

{\bf Acknowledgments.}
The author would like to thank Maria Axenovich for many discussions and suggestions regarding both the presentation and the content of this note.

\newpage

\appendix

\section{Appendix}

\begin{proof}[Proof of Lemma~\ref{lemma: coefficient}]

Suppose otherwise. 
Then there exists $z\in V$ such that, for every
$i\in\{0,\ldots,4\}$, both $v_{p_i}$ and $v_{q_i}$ have coefficient
one in $z$ with respect to the basis $\{v_x:x\in U_i\}$. 
By a classification of the copies of $C_{10}$ contained in a single
edge layer due to Axenovich, Martin, and Winter~\cite[Proof of Theorem~4]{AMW}, after relabeling the elements appearing in $U_0,\ldots,U_4$, these sets, in cyclic order, have one of the following forms: 
$\mathcal{H}_1=(ab,\allowbreak bc,\allowbreak cd,\allowbreak de,\allowbreak ea)$, $\mathcal{H}_2=(abc,\allowbreak bcd,\allowbreak cde,\allowbreak dea,\allowbreak eab)$, $\mathcal{H}_3=(cde,\allowbreak dea,\allowbreak aeb,\allowbreak ebc,\allowbreak bcd)$, $\mathcal{H}_4=(abc,\allowbreak bcd,\allowbreak cde,\allowbreak bde,\allowbreak bda)$, or $\mathcal{H}_5=(abcd,\allowbreak bcde,\allowbreak cdea,\allowbreak deab,\allowbreak eabc)$. 
Here, $a,b,c,d,e$ are distinct elements of $[n]$, and $abc$ denotes
$\{a,b,c\}$. 
Accordingly, $V$ has dimension $m = 2$ for $\mathcal{H}_1$, dimension $m = 3$ for $\mathcal{H}_2,\mathcal{H}_3,\mathcal{H}_4$, and dimension $m = 4$ for $\mathcal{H}_5$.

Write $\mathcal{H} = (U_0, \ldots, U_4)$, with subscripts taken modulo five. 
For every $i \in \{0, \ldots, 4\}$, let $K_i = U_i \cap U_{i+1}$. 
Then $|K_i| = m-1$ and $U_i = K_i \cup \{q_i\}$ and $U_{i+1}= K_i \cup \{p_{i+1}\}$. 
By assumption, $v_{q_i}$ has coefficient one in $z$ with respect to the basis $\{v_x: x \in U_i\}$, while $v_{p_{i+1}}$ has coefficient one with respect to the basis $\{v_x: x \in U_{i+1}\}$. 
Hence $z = v_{q_i} + k_i = v_{p_{i+1}} + k'_i$ for some $k_i, k'_i \in \operatorname{span}(\{v_x: x\in K_i\})$, and thus $v_{p_{i+1}} = v_{q_i} + (k_i - k'_i)$ and $k_i - k'_i \in \operatorname{span}(\{v_x: x \in K_i\})$.
Fix an ordered basis of $V$ and, for $x_1, \ldots, x_m\in V$, write $\det(x_1, \ldots, x_m)$ for the determinant of the matrix whose columns are the coordinate vectors of $x_1, \ldots, x_m$ in this basis. 
Take any ordering of the vectors indexed by $U_i$. 
Replacing the entry $v_{q_i}$ by $v_{p_{i+1}}$, while leaving the vectors indexed by $K_i$ in the same positions, gives an ordering of the vectors indexed by $U_{i+1}$ with the same determinant.  
Indeed, fix an ordering $K_i = \{x_1, \ldots, x_{m-1}\}$. 
By multilinearity of the determinant\footnote{That is, with all other
arguments fixed, the determinant is linear in each argument, i.e.,
$\det(\ldots, \alpha x + \beta y, \ldots) =\alpha\det(\ldots, x, \ldots) + \beta\det(\ldots, y, \ldots)$.}, we have $\det(v_{x_1},\ldots,v_{x_{j-1}},v_{p_{i+1}},v_{x_j},\ldots,v_{x_{m-1}}) = \det(v_{x_1},\ldots,v_{x_{j-1}},v_{q_i}, v_{x_j},\ldots,v_{x_{m-1}})+ 
\det(v_{x_1},\ldots,v_{x_{j-1}},k_i-k'_i, v_{x_j},\ldots,v_{x_{m-1}})
= \det(v_{x_1},\ldots,v_{x_{j-1}},v_{q_i},v_{x_j},\ldots,v_{x_{m-1}})$. 
The last equality holds because $k_i - k'_i \in \operatorname{span}(\{v_x: x \in K_i\})$, so the columns of the second determinant
are linearly dependent.

\caseheading\label{claim: case135} $\mathcal{H} \in \{\mathcal{H}_1, \mathcal{H}_3, \mathcal{H}_5\}$.

We shall use the preceding replacement argument repeatedly around the cyclic lists below. 
For $\mathcal{H}_1$, going once around the cyclic list gives $\det(v_a,v_b)
    =\det(v_c,v_b)
    =\det(v_c,v_d)
    =\det(v_e,v_d)
    =\det(v_e,v_a)
    =\det(v_b,v_a)
$.
For $\mathcal{H}_3$, it gives
$
    \det(v_c,v_d,v_e)
    =\det(v_a,v_d,v_e)
    =\det(v_a,v_b,v_e)
    =\det(v_c,v_b,v_e)
    =\det(v_c,v_b,v_d)
    =\det(v_c,v_e,v_d)
$.
Finally, for $\mathcal{H}_5$, it gives
$
    \det(v_a, \allowbreak v_b, \allowbreak v_c, \allowbreak v_d)
    =\det(v_e,v_b,v_c,v_d)
    =\det(v_e,v_a,v_c,v_d)
    =\det(v_e,v_a,v_b,v_d)
    =\det(v_e,v_a,v_b,v_c)
    =\det(v_d,v_a,v_b,v_c)
$.
In each case, the last ordering differs from the first by an odd permutation, so its determinant is the negative of the first. 
Since the first and last determinants in each chain are equal, the determinant is zero over $\mathbb{F}_3$, contradicting that the corresponding vectors form a basis.

\caseheading\label{claim: case2} $\mathcal{H} = \mathcal{H}_2$.

For $\mathcal{H}_2$, let $(\nu_0,\nu_1,\nu_2,\nu_3,\nu_4)=(v_a,v_b,v_c,v_d,v_e)$
with subscripts taken modulo five.
There are $s_0,\ldots,s_4\in\mathbb{F}_3$ such that $z=\nu_i+s_i\nu_{i+1}+\nu_{i+2}$ for every $i$.
The equations with indices $0$ and $1$ are $z=\nu_0+s_0\nu_1+\nu_2$ and $z=\nu_1+s_1\nu_2+\nu_3$, respectively, and hence $\nu_3 = \nu_0 + (s_0-1)\nu_1 + (1-s_1)\nu_2$. 
Using the equation with index $0$ to eliminate $z$ from the equation with index $2$, we obtain $\nu_4=\nu_0+s_0\nu_1-s_2\nu_3$. 
Substituting the preceding expression for $\nu_3$ gives $\nu_4 = (1-s_2)\nu_0 + (s_0 - s_2(s_0-1))\nu_1 + s_2(s_1-1)\nu_2$. 
On the other hand, using the equation with index $0$ to eliminate $z$ from the equation with index $4$ gives $\nu_4 = (1-s_4)\nu_0 + (s_0-1)\nu_1 + \nu_2$. 
Since $\nu_0,\nu_1,\nu_2$ form a basis, comparing the coefficients of $\nu_1$ and $\nu_2$ in these two expressions for $\nu_4$ gives $s_2(s_0-1) = 1$ and $s_2(s_1-1) = 1$, respectively. 
Hence $s_0 = s_1$. 
Repeating the same argument with the indices shifted in cyclic order gives $s_i = s_{i+1}$ for every $i$, and thus $s_0 = \ldots = s_4=s$. 
But then $s_2(s_0-1) = 1$ becomes $s(s-1) = 1$, which is impossible over $\mathbb{F}_3$.

\caseheading\label{claim: case4} $\mathcal{H} = \mathcal{H}_4$.

Finally, for $\mathcal{H}_4$, the two edges of $C$ incident with
$W\cup\{c,d,e\}$ have directions $c$ and $e$, while those incident with
$W\cup\{b,d,e\}$ have directions $b$ and $e$. 
Hence $z=v_c+\lambda v_d+v_e=v_b+\mu v_d+v_e$ for some
$\lambda,\mu\in\mathbb{F}_3$. 
Thus $v_b-v_c+(\mu-\lambda)v_d=0$, contradicting that
$v_b,v_c,v_d$ form a basis. 
This proves Lemma~\ref{lemma: coefficient}.

\end{proof}
    

\begin{thebibliography}{ACP+21}
    \providecommand{\url}[1]{\texttt{#1}}
    \providecommand{\urlprefix}{\textsc{url:} }
    \expandafter\ifx\csname urlstyle\endcsname\relax
      \providecommand{\doi}[1]{doi:\discretionary{}{}{}#1}\else
      \providecommand{\doi}{doi:\discretionary{}{}{}\begingroup \urlstyle{rm}\Url}\fi

\bibitem{axenovich} 
M.~Axenovich.
\newblock Extremal numbers for cycles in a hypercube.
Discrete Applied Mathematics 341 (2023), 1--3.

\bibitem{ABOT} 
M.~Axenovich, L.~Benz, D.~Offner, and C.~Tompkins.
\newblock Generalized Tur\'an densities in the hypercube.
Discrete Mathematics 346 (2023), no. 2, 113238.

\bibitem{AM}
M.~Axenovich and R.~R.~Martin.
\newblock A note on short cycles in a hypercube.
Discrete Mathematics 306 (2006), no. 18, 2212--2218.

\bibitem{AMW}
M.~Axenovich, R.~R.~Martin, and C.~Winter.
\newblock On graphs embeddable in a layer of a hypercube and their extremal numbers.
Annals of Combinatorics 28 (2024), no. 4, 1257--1283.

\bibitem{baber}
R.~Baber.
\newblock Tur\'an densities of hypercubes.
arXiv:1201.3587 (2012).

\bibitem{BHLL}
J.~Balogh, P.~Hu, B.~Lidick\'y, and H.~Liu.
\newblock Upper bounds on the size of 4- and 6-cycle-free subgraphs of the hypercube.
European Journal of Combinatorics 35 (2014), 75--85.

\bibitem{chung}
F.~Chung.
\newblock Subgraphs of a hypercube containing no small even cycles. 
Journal of Graph Theory 16 (1992), no. 3, 273--286.

\bibitem{conder}
M.~Conder.
\newblock Hexagon-free subgraphs of hypercubes.
Journal of Graph Theory 17 (1993), no. 4, 477--479.

\bibitem{conlon}
D.~Conlon.
\newblock An extremal theorem in the hypercube.
Electronic Journal of Combinatorics 17 (2010), no. 1, R111.

\bibitem{EIL}
D.~Ellis, M.-R.~Ivan, and I.~Leader.
\newblock Tur\'an densities for daisies and hypercubes.
Bulletin of the London Mathematical Society 56 (2024), no. 12, 3838--3853.

\bibitem{EKR}
P.~Erd\H{o}s, C.~Ko, and R.~Rado.
\newblock Intersection theorems for systems of finite sets.
Quarterly Journal of Mathematics, Oxford Series (2), 12 (1961), 313--320.

\bibitem{FOe1}
Z.~F\"uredi and L.~\"Ozkahya.
\newblock On even-cycle-free subgraphs of the hypercube.
Electronic Notes in Discrete Mathematics 34 (2009), 515--517.

\bibitem{FOe2}
Z.~F\"uredi and L.~\"Ozkahya.
\newblock On even-cycle-free subgraphs of the hypercube.
Journal of Combinatorial Theory, Series A 118 (2011), 1816--1819.

\bibitem{GM1} 
A.~Grebennikov and J.~P.~Marciano.
\newblock $C_{10}$ has positive Tur\'an density in the hypercube.
Journal of Graph Theory 109 (2025), no. 1, 31--34.

\bibitem{GM2}
A.~Grebennikov and J.~P.~Marciano.
\newblock $C_{10}$ has positive Tur\'an density in the hypercube.
arXiv:2402.19409v4 (2024; revised 2026).

\end{thebibliography}
\end{document}